\documentclass[a4paper,10pt]{amsart}
\usepackage{amsmath}
\usepackage{amsfonts}
\usepackage{amssymb}
\usepackage{amsthm}
\usepackage{hyperref}

\title{On Pellarin's $L$-Series}
\author{Rudolph Bronson Perkins}

\begin{document}
\newtheorem{thm}{Theorem}[subsection]
\newtheorem*{thm*}{Theorem}
\newtheorem{pr}[thm]{Proposition}
\newtheorem{lem}[thm]{Lemma}
\newtheorem{cor}[thm]{Corollary}
\newtheorem*{cor*}{Corollary}
\theoremstyle{remark}
\newtheorem{rem}[thm]{Remark}
\newtheorem{defn}[thm]{Definition}
\newtheorem{ex}[thm]{Example}

\begin{abstract}
Necessary and sufficient conditions are given for a negative integer to be a trivial zero of a new type of $L$-series recently discovered by F. Pellarin, and it is shown that any such trivial zero is simple. We determine the exact degree of the special polynomials associated to Pellarin's $L$-series. The theory of Carlitz polynomial approximations is developed further for both additive and $\mathbb{F}_q$-linear functions. Using Carlitz' theory we give generating series for the power sums occurring as the coefficients of the special polynomials associated to Pellarin's series, and a connection is made between the Wagner representation for $\chi_t$ and the value of Pellarin's $L$-series at 1.
\end{abstract}
\maketitle

\section{Introduction}
Let $A := \mathbb{F}_q[\theta]$, where $\theta$ is an indeterminate and $\mathbb{F}_q$ is the finite field with $q$ elements of characteristic $p$. Let $K_\infty := \mathbb{F}_q((\frac{1}{\theta}))$ be the completion of $K := \mathbb{F}_q(\theta)$, the fraction field of $A$, equipped with the absolute value $|\cdot|$ normalized so that $|\theta| = q$. Let ${\mathbb C}_\infty$ be the completion of the algebraic closure of $K_\infty$ equipped with the canonical extension of $|\cdot|$. For positive integers $s$ and for $t$ in ${\mathbb C}_\infty$ such that $|t|< |\theta|$, Pellarin defines the following characteristic $p$ valued $L$-series \cite{Pe-11},
$$L(\chi_t, s) := \sum_{a \in A_+} \chi_t(a) a^{-s},$$
where $A_+$ is the space of all monic polynomials in $A$ and $\chi_t$ is the $\mathbb{F}_q$-algebra morphism of evaluation mapping $\theta$ to $t$. Remarkably, using the theory of deformations of vectorial modular forms he is able to associate this $L$-series at certain positive integers to the function 
$$\omega(t) := (-\theta)^{\frac{1}{q-1}}\prod_{i = 0}^\infty \left( 1 - \frac{t}{\theta^{q^i}}\right)^{-1}$$ of Anderson and Thakur \cite{AT-90} (denoted there by $\omega_1$). This function is defined by this formula for $t \in \mathbb{C}_\infty$ such that $|t|<|\theta|$ and depends on a choice of $(q-1)$-th root of $-\theta \in \mathbb{C}_\infty$. This establishes a long desired connection between $L$-series and modular forms in positive characteristic and also between $L$-series and the function $\omega(t)$. These functions $L(\chi_t,s)$ may be analytically continued, in the rigid sense, in the variable $t$ to all of $\mathbb{C}_\infty$, and for $t = \theta$ and certain positive integers $s$ Pellarin's results specialize to the computations of Carlitz from the 1930's of the values of the Goss-Carlitz zeta function at the ``even'' positive integers. 

Because of these connections, Goss was lead to show that Pellarin's series fit into his very general theory of $L$-series which we now briefly describe. Let $\pi  = \theta^{-1}$ be a uniformizer at $\infty$, and for $a \in A$ we define $\left< a \right> := \pi^{\deg(a)}a$, which is a $1$-unit in $K_\infty$. Let $A_+(d)$ be the subset of $A$ of monic polynomials of degree $d$. In \cite{Go-11}, Goss defines the formal series
$$L(\chi_t^\beta,x,y) := \sum_{d = 0}^\infty x^{-d} \sum_{a \in A_+(d)} \chi_t(a)^\beta\left< a \right>^{-y},$$ and shows for $\beta = 1$ that these series give a continuous family of entire power series in the variables $x,t \in \mathbb{C}_\infty$ parameterized by $y \in \mathbb{Z}_p$. These properties also hold for all positive integers $\beta$, see \cite{Pe-11} Remark 7. To these series he also associates the \emph{special polynomials} at the non-positive integers, 
$$z(\chi_t^\beta, x, -k) := L(\chi_t^\beta,x\pi^{k},-k) = \sum_{e = 0}^\infty x^{-e} \left( \sum_{a \in A_+(e)} \chi_t(a)^\beta a^{k} \right),$$
where, by \cite{Go-11}, the sum in parenthesis vanishes for $e \gg 0$. It is our aim to study these special polynomials when $\beta > 0$, and to prove results which are analogous to what is known for the special polynomials associated to Goss' original extension of Carlitz' zeta function. 

We give a short review of the results concerning the Goss-Carlitz zeta function which are most relevant to what we shall prove in this note. Since Carlitz' in the 30's several people have studied the \emph{power sums}, defined for non-negative integers $k$ and $d$ by
$$S_d(k) := \sum_{a \in A_+(d)}a^k.$$
The next theorem which gives necessary and sufficient conditions for the non-vanishing of the power sums $S_d(k)$ was known to Carlitz but given proof many years later by Sheats \cite{Sh98}.

\begin{thm*}[Sheats \cite{Sh98}, Theorem 1.4 (a)]
The sum $S_d(k)$ is non-zero if and only if there exists a $(d+1)$-tuple $(i_0,i_1,\dots,i_d)$ of non-negative integers such that
\begin{enumerate}
\item $\sum_{j = 0}^d i_j = k$,
\item There is no carry over of $p$-adic digits in the sum $\sum_{j = 0}^d i_j$, and
\item $i_j > 0$ and $(q - 1) \ | \ i_j$ for $0 \leq j \leq d-1$.
\end{enumerate}
\end{thm*}

Sheats' theorem is one of the main ingredients in the proof of the Riemann Hypothesis for Goss' zeta function associated to $\mathbb{F}_q[\theta]$ for general $q$. B\"ockle used this theorem to give an exact degree in $x^{-1}$ of the special polynomials associated to Goss' zeta function
$$z(x,-k) := \sum_{d = 0}^\infty x^{-d} S_d(k).$$

Let $\log_p$ denote the logarithm in base $p$, and for real numbers $x$, let $\lfloor x \rfloor$ denote the greatest integer less than or equal to $x$. For positive integers $\alpha$, let $l(\alpha)$ be the sum of the base-$q$ digits of $\alpha$.

\begin{cor*}[B\"ockle \cite{Bo-12}, Theorem 1.2 (a)]
The degree in $x^{-1}$ of $z(x,-k)$ equals 
$$\min_{0 \leq i < \log_p(q)} \left\lfloor \frac{l(p^ik)}{q-1} \right\rfloor.$$
\end{cor*}

Using a specialization argument suggested by B\"ockle we are able to prove in Theorem \ref{exactdegree} that the degree in $x^{-1}$ of the special polynomial $z(\chi_t^\beta,x,-k)$ is 
$$\min_{0 \leq i < \log_p(q)} \left\lfloor \frac{l(p^i \beta) + l(p^i k)}{q-1} \right\rfloor.$$
One easily sees that $l(k) \leq (q-1)(\log_q(k)+1)$, and for fixed $\beta$ this implies the logarithmic growth of the degrees in $x^{-1}$ of the special polynomials associated to Goss' zeta function.

In \cite{Bo-02}, B\"ockle showed under very general circumstances that the degrees of the special polynomials associated to $L$-series of Drinfeld modules grow logarithmically. Using this logarithmic growth Goss \cite{Go-05} was then able to give continuous-analytic continuation in a sense similar to what we have described above for Pellarin's series to these more general $L$-series via methods of integration in positive characteristic. 

A large portion of this paper is devoted to developing Carlitz' theory of polynomial interpolations of $\mathbb{F}_q$-linear functions on the sets $A_+(d)$. As a first application of this theory, for certain positive integers $\beta$, we are able to give a sufficient condition by means of generating series for the vanishing of the multivariate power sums
 $$S_d(\chi_t^\beta,k) := \sum_{a \in A_+(d)} \chi_t(a)^\beta a^{k}$$
 which appear as the coefficients of the powers of $x^{-1}$ in the special polynomials associated to Pellarin's $L$-series, and this condition is strong enough to recover the logarithmic growth of the degrees in $x^{-1}$ of the polynomials $z(\chi_t^\beta,x,-k)$ for these certain $\beta$. 

As a second application, we also use Carlitz' interpolation polynomials in Theorem \ref{chicoeffs} to show that for all non-negative integers $d$ we have
$$\sum_{a \in A_+(d)}\frac{\chi_t(a)}{a} = (-1)^{d}\frac{b_d(\chi_t)}{L_d},$$
where $L_i$ is the least common multiple of all polynomials of degree $i$ in $A$ and $b_d(\chi_t)$ is the $d$-th Wagner coefficient for the $\mathbb{F}_q$-linear function $\chi_t$ defined in Section \ref{sec:wag}. This turns out to be related to an important connection between the Carlitz logarithm and the value $L(\chi_t,1)$ which was known to F. Pellarin and developed in Section 4 of \cite{Pe-11}. In a future work of the author we find closed form expressions for the degree $d$ part of values of a slightly generalized version of Pellarin's $L$-series at various positive integers using these interpolation polynomials.

First we present a recursive formula for $z(\chi_t^\beta, x, -k)$. This will show, as Goss does in his note, that at the negative integers these series actually live in $A[t][x^{-1}]$ and, upon substituting $x = 1$, will give necessary and sufficient conditions for a negative integer $-k$ to be a zero (called by Goss a \emph{trivial zero}) of Pellarin's $L$-series. Furthermore, for such negative integers $-k$ we will show that the zeros at $x = 1$ of $z(\chi_t^\beta,x,-k)$ are simple. 

The familiar reader will surely see both Carlitz' and Goss' techniques sprinkled throughout this work. It is the author's great pleasure to thank David Goss for suggesting the topic of this paper and for his enthusiasm which fueled much of this work. His many comments and suggestions have added much of value to both the author's and this paper's development. The author would also like to thank the referee for many suggestions which helped improve this paper. Finally, I thank Chelsea Perkins for her constant hard work and selflessness which gives me the freedom to do such enjoyable things.

\section{Recursive Formulas and Trivial Zeros}
The notation from the introduction remains in effect. In all that happens below $k$ and $\beta$ will be arbitrary non-negative integers, $t$ an indeterminate. 

\begin{pr} \label{pr:zrecurs}
Let $\beta$ and $k$ be non-negative integers. Let 
$$X := \{(\alpha, l) \in \mathbb{Z}^2 : 0 \leq \alpha \leq \beta, 0 \leq l \leq k, (\alpha,l) \neq (\beta, k) \text{ and } (q-1) | (\beta + k -(\alpha + l)) \}.$$ 
Then the following recursive formula holds: $z(\chi_t^0, x, 0) = 1$, and
\begin{eqnarray} \label{eqn:zrecurs}
z(\chi_t^\beta,x,-k) = 1 - x^{-1} \sum_{(\alpha,l) \in X}{\beta \choose \alpha}{k \choose l}t^\alpha \theta^k  z(\chi_t^\alpha, x, -l).
\end{eqnarray}
\begin{proof}
This is an elementary calculation built on the observation that each element $a$ in $A_+(d)$ may be obtained from an element $h$ in $A_+(d-1)$ and an element $b$ in $\mathbb{F}_q$ by writing $a = \theta h + b$ and $a(t) = t h(t)+b$, where here we denote the image of $\chi_t(a)$ by $a(t)$ for all $a \in A$. One then uses the binomial theorem to decompose the summand. Interchanging sums and counting finishes the proof.
\end{proof}
\end{pr}

\begin{cor} \label{cor:lrecurs}
Let $X$ be as in Proposition \ref{pr:zrecurs}. Then $L(\chi_t^0, 0) = 1$, and
\begin{eqnarray} \label{eqn:lrecurs}
L(\chi_t^\beta,-k) = 1 - \sum_{(\alpha,l) \in X}{\beta \choose \alpha}{k \choose l}t^\alpha \theta^k  L(\chi_t^\alpha, -l).
\end{eqnarray}
\begin{proof}
Substituting $x = 1$ into \eqref{eqn:zrecurs} gives the result.
\end{proof}
\end{cor}

In \cite{Go-11}, Goss showed that the $L$-values $L(\chi_t^\beta,-k)$ vanish for positive integers $k$ congruent to $-\beta \mod (q-1)$, and due to a link that exists between these $L$-values and classical characteristic zero $L$-values, see \cite{Go96} Section 4.13, he refers to these as \emph{trivial zeros of Pellarin's $L$-series}. As these $L$-values are obtained from the special polynomials $z(\chi_t^\beta,x,-k)$ by evaluation at $x = 1$ we shall say that $L(\chi_t^\beta,-k)$ is a \emph{simple zero of Pellarin's $L$-series} if $z(\chi_t^\beta,x,-k)$ has a simple zero at $x = 1$.

\begin{thm}
Let $\beta$ be a fixed non-negative integer. Then Pellarin's $L$-series $L(\chi_t^\beta, s)$ vanishes for each negative integer $s = -k$ such that $\beta + k \equiv 0 \text{ mod }(q-1).$ If $\beta + k \not\equiv 0 \text{ mod } (q-1)$ then $L(\chi_t^\beta, s)$ does not vanish. Furthermore, each trivial zero of Pellarin's $L$-series is simple.
\begin{proof} 
The first pair of non-negative integers $(\beta,k)$ for which the sum in \eqref{eqn:lrecurs} is non-empty is when $\beta + k = q-1$. In this case, the first non-zero term in the series occurs when $(\alpha,l) = (0,0)$, and this corresponds to $L(\chi_t^0,0) = 1$. Hence $L(\chi_t^\beta, -k) = 0$. 

Let $k$ and $\beta$ be non-negative integers, and suppose that $\beta + k \equiv 0 \text{ mod } (q-1)$. Then the only terms which occur in the recursive formula for $L(\chi_t^\beta, -k)$ correspond to pairs $(\alpha, l)$ such that either $\alpha < \beta$ or $l < k$, and $\alpha + l \equiv 0 \text{ mod }(q-1)$. In particular, the summand corresponding to the pair (0,0) appears and this eliminates the $1$ on the right side of \eqref{eqn:lrecurs}. By induction on the pairs $(\alpha,l) \neq (0,0)$ we see that each $L(\chi_t^\alpha, -l)$ vanishes. As these are precisely the terms remaining in \eqref{eqn:lrecurs}, we conclude that $L(\chi_t^\beta, -k)$ vanishes.

Conversely, if $\beta$ and $k$ are such that $\beta + k \not\equiv 0 \text{ mod }(q-1)$, then the summand corresponding to the pair $(0,0)$ does not appear. Thus $L(\chi_t^\beta,s)$ has constant coefficient equal to $1$ and does not vanish. 

To see that the above zeros are simple we examine the derivative of $z(\chi_t^\beta,x,-k)$ with respect to $x$. We have
\begin{eqnarray} \label{eqnarray:simpzero}
&& \frac{\partial}{\partial x}(z(\chi_t^\beta,x,-k)) = \\
&=& x^{-2}\sum_{(\alpha,l) \in X} {\beta \choose \alpha}{k \choose l}t^{\alpha}\theta^{k}z(\chi_t^\alpha,x,-l) - \\
&&-x^{-1}\sum_{(\alpha,l) \in X} {\beta \choose \alpha}{k \choose l}t^{\alpha}\theta^{k}\frac{\partial}{\partial x}(z(\chi_t^\alpha,x,-l)).
\end{eqnarray}
Now suppose $\beta+k = (q-1)$. Then the only term appearing in either sum corresponds to the pair $(0,0)$. In the first sum this is just $x^{-2}$. In the second sum we have $(\partial / \partial x)(z(\chi_t^0,x,0)) = (\partial / \partial x)(1) = 0$. Hence, substituting $x = 1$ shows that the derivative of the special polynomial does not vanish. If $\beta + k \equiv 0 \text{ mod }(q-1)$ the term corresponding to the pair $(0,0)$ appears again giving rise to an $x^{-2}$. The only other terms which may appear are those corresponding to pairs $(\alpha, l)$, where either $\alpha > 0$ or $l > 0$ and $\alpha + l \equiv 0 \mod (q-1)$. Suppose for illustration that $\alpha > 0$. We are looking at
$$x^{-2}{\beta \choose \alpha}{k \choose l}t^{\alpha}\theta^{k}z(\chi_t^\alpha,x,-l) -x^{-1}{\beta \choose \alpha}{k \choose l}t^{\alpha}\theta^{k}\frac{\partial}{\partial x}(z(\chi_t^\alpha,x,-l)).$$
By induction, $(\partial / \partial x)(z(\chi_t^\alpha,x,-l))$ does not vanish upon evaluation at $x = 1$, and thus the second term above is divisible by $t^\alpha$ (hence by $t$). The first term above vanishes when evaluated at $x = 1$ as has already been shown.

Thus in general we have, $(\partial / \partial x)(z(\chi_t^\beta,x,-k)) = x^{-2} + t f(\theta,t,x^{-1}) + \theta g(\theta, t, x^{-1}),$ 
for two polynomials $f, g \in \mathbb{F}_q[\theta, t, x^{-1}]$ both of which may be zero (as in the case $\beta + k = q-1$). Thus $(\partial / \partial x) (z(\chi_t^\beta,x,-k))|_{x = 1}$ does not vanish. 
\end{proof}
\end{thm}

\section{Logarithmic Growth}
Throughout this section let $t$ be an indeterminate. We recall the definition of the special polynomials associated to Pellarin's $L$-series. For $\beta,k$ any non-negative integers we defined
$$z(\chi_t^\beta, x, -k) := \sum_{d = 0}^\infty x^{-d} \left( \sum_{a \in A_+(d)} \chi_t(a)^\beta a^{k} \right),$$
and we denoted the coefficient of $x^{-d}$ in $z(\chi_t,x,-k)$ by
$$S_d(\chi_t^\beta,k) := \sum_{a \in A_+(d)} \chi_t(a)^\beta a^k.$$ 

For positive integers $\alpha = a_d q^d + a_{d-1}q^{d-1} + \cdots + a_1 q + a_0$ written in base $q$ with $a_d \neq 0$ define $\deg_q(\alpha) = d$, and define the \emph{$q$-length} of $\alpha$ to be $l(\alpha) := \sum_{i = 0}^d a_d.$

Finally, for a pair of positive integers $(\beta,k)$ define
$$\phi(\beta,k):= \min_{0\leq i <\log_p(q)}\left\lfloor \frac{l(p^i\beta)+l(p^i k)}{q-1} \right\rfloor.$$

\subsection{Exact Degree for $z(\chi_t^\beta,x,-k)$}

In this section we calculate the exact degree in $x^{-1}$ of the special polynomials associated to Pellarin's $L$-series by means of a specialization argument suggested to us by B\"ockle.

\begin{thm}\label{exactdegree}
Let $\beta,k$ be positive integers. Then the exact degree in $x^{-1}$ of the polynomial $z(\chi_t^\beta,x,-k)$ is $\phi(\beta,k)$.
\begin{proof}
We view $S_d(\chi_t^\beta,k)$ as a polynomial in $t$ over $A$. We will show that $S_d(\chi_t^\beta,k)$ vanishes for $t = \theta^{q^m}$ for all $m$ sufficiently large, and hence vanishes identically in $A[t]$.

Let $M = \max_{0\leq i <\log_p(q)}\{\deg_q(p_i k)\}$. Let $m > M$. Then for $0 \leq i < \log_q(p)$ there is no carry-over of digits in base $q$ in the sum $q^{m}p^i\beta + p^i k$. Hence 
$$l(q^{m}p^i\beta + p^i k) = l(q^{m}p^i\beta) + l(p^i k) = l(p^i\beta)+l(p^i k).$$
Thus
$$\phi(\beta,k) = \min_{0 \leq i <\log_p(q)}\left\lfloor \frac{l(q^{m}p^i\beta + p^i k)}{q-1} \right\rfloor.$$
Then, by B\"ockle's theorem \cite{Bo-12}, Theorem 1.2 (a), for $d > \phi(\beta,k)$ and all $m > M$ we have
$$S_d(\chi_t^\beta,k)|_{t = \theta^{q^m}} = S_d(q^{m}\beta + k) = 0.$$
Hence $S_d(\chi_t^\beta,k)$ is identically zero whenever $d > \phi(\beta,k)$.

Similarly, for $d = \phi(\beta,k)$ and any $m > M$ we have
$$S_d(\chi_t^\beta,k)|_{t = \theta^{q^m}} = S_d(q^{m}\beta + k) \neq 0.$$
We conclude that $S_d(\chi_t^\beta,k) \neq 0$ for $d = \phi(\beta,k)$.

It follows that the degree in $x^{-1}$ of $z(\chi_t^\beta,k)$ is exactly $\phi(\beta,k)$ as claimed. 
\end{proof}
\begin{rem}
Should it become necessary to examine the power sums 
$$S_d(\beta_1,\dots,\beta_r,k) := \sum_{a \in A_+(d)}\chi_{t_1}(a)^{\beta_1}\cdots\chi_{t_r}(a)^{\beta_r}a^k \in A[t_1,\dots,t_r]$$ and the special polynomials associated to them
$$z_d(\beta_1,\dots,\beta_r,x,k) = \sum_{d = 0}^\infty x^{-d}S_d(\beta_1,\dots,\beta_r,k)$$
the argument above extends inductively to show that these are indeed polynomials and to determine the exact degree in $x^{-1}$ of these multi-variate special polynomials. We leave the details to the reader.
\end{rem}
\end{thm}

\section{Carlitz Approximations}
\subsection{Development and Logarithmic Growth}
In this section we develop the theory of Carlitz polynomial approximations. As a first application we will prove the theorem stated just below. In light of the previous section, one also wants some other application of these Carlitz polynomials. In the following section we connect them to the value of Pellarin's series $L(\chi_t,1)$.

\begin{thm}\label{thm:loggrowth}
Let $\beta$ be a fixed non-negative integer. Suppose $l(\beta) < q$. Then:
\begin{enumerate}
\item The polynomials $S_d(\chi_t^\beta, k)$ vanish when $k < q^d - q^{d-1}l(\beta) - 1$.

\item Let $k$ be a fixed positive integer. Then the degree of $z(\chi_t^\beta,x,k)$ in $x^{-1}$ is at most the greatest integer less than $\log_q((k+1)/(q - l(\beta))) + 1$.
\end{enumerate} 
\end{thm}

We begin with a definition.

\begin{defn}
For a ring $L$ of \emph{characteristic} $p$ we define the \emph{affine polynomial ring associated to} $L$ as a set by
$$\mathcal{A}(L) := \{l_{-1} + \sum_{i = 0}^n l_i z^{q^i} : n \geq 0 \text{ and } l_i \in L \text{ for all } i = -1,0,1,...\} \subseteq L[z].$$
The set $\mathcal{A}(L)$ forms a ring with the operations of polynomial composition and addition. Polynomials $f \in \mathcal{A}(L)$ will be called \emph{affine}. 

The subring of $\mathcal{A}(L)$ consisting of those polynomials for which $l_{-1} = 0$ will be called \emph{the ring of $\mathbb{F}_q$-linear polynomials}. We will refer to elements in this subring as \emph{$\mathbb{F}_q$-linear}.
\end{defn} 

For $d \geq 1$, let $A(d)$ be the $\mathbb{F}_q$-vector space of polynomials in $A$ of degree strictly less than $d$. Let $A(0) = \{0\}$. For $d \geq 1$, let $D_d$ be the product of all monic polynomials of degree $d$, and $L_d$ be the least common multiple of all polynomials of degree $d$. Let $D_0 = L_0 = 1$. For $d \geq 0$, Carlitz introduced and studied the polynomials
$$e_d(z) := \prod_{a \in A(d)}(z - a).$$
We will need some of their basic properties to follow.
\begin{thm}[Carlitz] \label{carlitzcoeff}
For $d \geq 0$, the polynomial $e_d(z)$ is $\mathbb{F}_q$-linear, and the coefficient of $z$ is equal to $(-1)^d D_d/ L_d$.
\begin{proof}
See Goss \cite{Go96}, Chapter 3, Section 1.
\end{proof}
\end{thm}

We fix an $A$-algebra $B$ which is an integral domain. Next we will define and discuss a family of $B$-linear operators (first studied by Carlitz \cite{Ca-40}) on the space of functions from $A$ to $B$.

\begin{defn}
Let $f$ be a function from $A$ to $B$, and let $d \geq 0$. The \emph{Carlitz approximation} to $f$ on $A_+(d)$ is the function
$$M_d(f)(z) := \sum_{b \in A_{+}(d)}\left(f(b)\prod_{a \in A_{+}(d) \setminus \{b\}}(z - a)\right).$$ 
\end{defn}
\begin{rem}
The quantity $M_d(1)(z)$ is the derivative with respect to $z$ of 
$$e_d(z - \theta^d) = \prod_{a \in A_+(d)}(z-a).$$ 
By \ref{carlitzcoeff} this is equal to $(-1)^d D_d / L_d$. For now we will denote this constant by $M_d(1)$.
\end{rem}

The following lemma from basic algebra will appear often, and we state it now so that it may be quoted more briefly in all that follows.
\begin{lem}\label{degreelemma}
Any polynomial of degree $d$ with coefficients in an integral domain $R$ has at most $d$ distinct roots in $R$.
\end{lem}

Observe that the $M_d(f)(z) \in B[z]$ are simply interpolation polynomials similar to those of Newton or Lagrange. We have the following theorem due originally to Carlitz.
\begin{thm}
Let $f$ be a function from $A$ to $B$. The polynomial $M_d(f)(z)$ is the unique polynomial in $B[z]$ of degree strictly less than $q^d$ that agrees with $M_d(1)f$ on $A_+(d)$.
\begin{proof}
We may rewrite $M_d(f)(z)$ more compactly as
$$M_d(f)(z) = \sum_{a \in A_+(d)}f(a)\frac{e_d(z - a)}{z - a}.$$
The polynomial $e_d(z)$ is $\mathbb{F}_q$-linear, and hence its formal derivative with respect to $z$ is just a constant in $A$, which we called $M_d(1)$ above. Let $b \in A_+(d)$. Evaluation of the quotient $e_d(z - b)/(z-b)$ at $z = b$ equals the formal derivative of $e_d(z - b)$ at $z = b$, i.e. the constant $M_d(1)$. Hence
$$M_d(f)(b) = \sum_{a \in A_+(d)}f(a)\left.\frac{e_d(z - a)}{z - a}\right|_{z = b} = M_d(1)f(b).$$

Uniqueness follows by Lemma \ref{degreelemma}.
\end{proof}
\end{thm}

\begin{rem}
As motivation for the rest of this section, we describe how Carlitz approximation polynomials were rediscovered by the author. Carlitz originally observes \cite{Ca-35} that the logarithmic derivative of 
$$e_d(z - \theta^d) := \prod_{a \in A_+(d)}(z-a)$$ 
gives a generating series for the numbers $S_d(k)$ as $k$ varies over the non-negative integers. Explicitly,
\begin{eqnarray} \label{eqnarray:carlitz}
\frac{\frac{\partial}{\partial z}(e_d(z - \theta^d))}{e_d(z - \theta^d)} = \sum_{e = 0}^\infty \frac{S_d(e)}{z^{e+1}}.
\end{eqnarray}
By Theorem \ref{carlitzcoeff}, 
$$\frac{\partial}{\partial z}(e_d(z - \theta^d)) = (-1)^d\frac{D_d}{L_d}.$$
Combining this with the easy observation that the degree in $z$ of $e_d(z)$ is $q^d$, we see that the order with which $z^{-1}$ divides this logarithmic derivative is $q^d$, and the logarithmic growth follows. 

In an interesting twist, replacing the derivative of $e_d(z - \theta^d)$ in \eqref{eqnarray:carlitz} by the $d$-th Carlitz approximation of $\chi_t(z)$ on $A_+(d)$, namely
$$M_d(\chi_t)(z):= \sum_{a \in A_+(d)}\chi_t(a)\frac{e_d(z - a)}{z-a},$$ 
yields a generating series for the $S_d(\chi_t, k)$ as $k$ varies over the positive integers. We will show that $M_d(\chi_t)(z)$ lies in the ring $Aff(A[t])$. This is enough to obtain the logarithmic growth in $z^{-1}$ of the degrees of the special polynomials in the case $\beta = 1$. For the more general case, we refer the reader to the proof below. 

One can consider the logarithmic derivative of the following formal symbol in order to obtain a generating series for the $S_e(\chi_t^\beta, k)$ as $k$ varies over the positive integers. No effort has been made to understand or even define the following object. We introduce the formal symbol, the ``$\chi_t^\beta$-th \emph{exponential}'' for $A_+(d)$,
$$E_d(\chi_t^\beta)(z) := \prod_{a \in A_+(d)}(z-a)^{\chi_t(a)^\beta}.$$
Formally the logarithmic derivative  $d\log/dz$ of $E_d(\chi_t^\beta)(z)$ with respect to $z$ is then
$$\frac{d\log}{dz}(E_d(\chi_t^\beta)(z)) = \sum_{a \in A_+(d)} \frac{\chi_t(a)^\beta}{(z-a)}.$$
After finding common denominators, the numerator of our result is $M_d(\chi_t^\beta)$. For more on Carlitz' polynomial approximations  and their limits, see \cite{Ca-40}, \cite{Go-89}, \cite{Wa-71}. 
\end{rem}

We believe the next observation is new to the theory. It allows us to use $M_d(f)$ to interpolate the values of $f$ on $A(d)$ for $f$ in the space $\text{Hom}_+(A, B)$ of additive group homomorphisms from $A$ to $B$. We hope to explore this connection further in a future work. For our present purposes the affineness of the polynomial $M_d(f)$ will follow for $\mathbb{F}_q$-linear functions $f$.

\begin{thm} \label{pr:fundrel}
For all $f \in \text{Hom}_+(A,B)$, for all positive integers $d$, and for all $c \in A(d)$ we have
\begin{eqnarray} \label{eqn:fundrel}
M_d(f)(z+c) - M_d(f)(z) = M_d(1)f(c).
\end{eqnarray}
\begin{proof}
 Let $c \in A(d)$. Then we have
\begin{eqnarray*}
 M_d(f)(z+c) &=& \sum_{b \in A_+(d)}\left(f(b)\prod_{a \in A_+(d) \setminus \{b\}}(z +c - a)\right) \\
 &=& \sum_{b \in A_+(d)}\left(f(b - c + c)\prod_{a \in A_+(d) \setminus \{b\}}(z - (a-c))\right) \\
  &=& \sum_{b \in A_{+}(d)}\left((f(b - c) + f(c))\prod_{a \in A_{+}(d) \setminus \{b\}}(z - (a-c))\right) \\
&=& \sum_{b \in A_{+}(d)}\left(f(b-c)\prod_{a \in A_{+}(d) \setminus \{b\}}(z - (a-c))\right) + \\
& & + \ f(c)\sum_{b \in A_{+}(d)}\prod_{a \in A_{+}(d) \setminus \{b\}}(z - (a-c)).\end{eqnarray*}
The map from $A_{+}(d)$ to $A_{+}(d)$ sending $a$ to $a+c$ for all $a \in A_{+}(d)$ is a bijection of sets. Thus as $a$ runs over all elements of $A_+(d)\setminus \{b\}$, $a-c$ runs over all elements of $A_+(d)\setminus \{b-c\}$, and similarly for $b-c$. Making the change of variables $u =a-c$ and $w = b-c$ in the last line above and re-indexing accordingly gives
$$\sum_{w \in A_{+}(d)}\left(f(w)\prod_{u \in A_{+}(d) \setminus \{w\}}(z - u)\right) + f(c)\sum_{w \in A_{+}(d)}\prod_{u \in A_{+}(d) \setminus \{w\}}(z - u),$$
which is equal to $M_d(f)(z) + M_d(1)f(c)$.
\end{proof}
\end{thm}

\begin{defn}
We denote the space of $\mathbb{F}_q$-linear maps from $A$ to $B$ by $L_{\mathbb{F}_q}(A, B)$.
\end{defn}
\begin{cor} \label{cor:linear}
Let $f \in L_{\mathbb{F}_q}(A, B)$. Let $P_d(f)(z) := M_d(f)(z) - M_d(f)(0)$. Then 
\begin{enumerate}
\item $P_d(f)(c) = M_d(1)f(c)$ for all $c \in A(d)$, and $P_d(f)(z)$ is the unique polynomial in $B[z]$ of degree strictly less than $q^d$ with this property;
\item $P_d(f)(z)$ is $\mathbb{F}_q$-linear in $z$.
\end{enumerate}
\begin{proof}
Substituting $z = 0$ into \eqref{eqn:fundrel}, we see that $P_d(f)(c)$ agrees with $M_d(1)f(c)$ for all $c$ in $A(d)$ (recall that $M_d(1)$ is a constant in $A$). It is clear from definitions that the degree of $P_d(f)(z)$ is at most $q^d - 1$. Hence uniqueness follows from \ref{degreelemma}.

Using \ref{degreelemma}, a similar argument on the number of roots of $P_d(f)(z+c) - P_d(f)(z) - P_d(f)(c)$, first for $c \in A(d)$ then for $c$ and indeterminate, shows that that $P_d(f)$ is additive. Likewise for $\mathbb{F}_q$-linearity.
\end{proof}
\end{cor}

\begin{cor} \label{cor:affine}
Let $f \in L_{\mathbb{F}_q}(A, B)$. Then the polynomial $M_d(f)(z)$ is affine and has degree at most $q^{d-1}$. \begin{proof}
The polynomials $M_d(f)(z)$ and $P_d(f)(z)$ differ by an element of $B$. Hence $M_d(f)(z)$ is affine and their degrees are equal. 
\end{proof}
\end{cor}
\begin{rem}
It is important to observe that $M_d(f)(0) \neq 0$ in general. Indeed, this will be seen in the case of $M_d(\chi_t)$.
\end{rem}

The next corollary is special in its use of the \emph{digit principal} which is basic to the domain of function field arithmetic. In it we prove that the operators $M_d$ are ``multiplicative" on functions $f^\beta$ for $\mathbb{F}_q$-linear functions $f$ and certain special $\beta$.
\begin{cor} \label{cor:digit}
Let $f \in L_{\mathbb{F}_q}(A, B)$, and let $\beta$ be a non-negative integer as above. Suppose that $l(\beta) < q$. Then
$$M_d(f^\beta)(z) = M_d(1)^{1 - l(\beta)}\prod_{i = 0}^e \left( M_d(f^{q^i})(z) \right)^{\beta_i}.$$ 
\begin{proof}
The ring $B$ is a domain, and hence $M_d(1)^{-1} M_d(f)(z)$ is the unique polynomial over $B$ of degree less than $q^d$ which agrees with $f(z)$ on $A_+(d)$. By \autoref{cor:affine}, for all $i$ the polynomial approximation $M_d(f^{q^i})(z)$ has degree at most $q^{d-1}$. Write $f(z)^\beta = \prod_{i = 0}^e (f(z)^{q^i})^{\beta_i}$. Replacing $f(z)^{q^i}$ with its corresponding polynomial representation 
 $$M_d(1)^{-1} M_d(f^{q^i})(z)$$ gives a polynomial approximation of degree at most $q^{d-1} l(\beta)$ which agrees with $f(z)^\beta$ on $A_+(d)$. Hence if $l(\beta) < q$, then $q^{d-1} l(\beta) < q^d$ and Lemma \ref{degreelemma} implies 
 $$M_d(1)^{-1} M_d(f^\beta)(z) = \prod_{i = 0}^e [M_d(1)^{-1} M_d(f^{q^i})(z)]^{\beta_i}.$$ 
\end{proof}
\end{cor}

\begin{proof}[Proof of Theorem \ref{thm:loggrowth}]
 Suppose $l(\beta) < q$. For the first claim, expand 
\begin{eqnarray} \label{eqn:genfn}
\frac{M_d(\chi_t^\beta)(z)}{\prod_{a \in A_{+}(d)}(z-a)} = \sum_{a \in A_+(d)}\frac{\chi_t(a)^\beta}{z-a}
\end{eqnarray} 
in a power series in the variable $1/z$ using geometric series. The coefficient of $1/z^{k+1}$ is $S_d(\chi_t^\beta,k)$. By \autoref{cor:digit}, the order with which $1/z$ divides the left hand side of \eqref{eqn:genfn} is $q^{d} - q^{d-1}l(\beta)$. Hence the first non-zero coefficient of the corresponding power series occurs only when $k \geq q^d - q^{d-1}l(\beta) - 1$.
 
 For the second claim, observe that $S_e(\chi_t^\beta, k)$ is the coefficient of $x^{-e}$ in $z(\chi_t^\beta, x, k)$.
\end{proof}

\subsection{Wagner Coefficients and $L(\chi_t,1)$}\label{sec:wag}
Recall that we defined $\mathbb{C}_\infty$ to be the completion of an algebraic closure of $K_\infty := \mathbb{F}_q((\frac{1}{\theta}))$ equipped with the unique extension of the absolute value $|\cdot|$ making $K_\infty$ complete and normalized so that $|\theta| = q$. 

To the function $\chi_t(z)$ we may associate a formal series which we call its \emph{Wagner series}:
$$\sum_{j = 0}^\infty b_{j}(\chi_t) \frac{e_{j}(z)}{D_j} \in K[[z,t]],$$
where the coefficients $b_j(\chi_t)$ are defined for $j \geq 0$ by
$$b_j(\chi_t) := (-1)^j\frac{L_j}{D_j}M_j(\chi_t)(0).$$
\begin{rem}
At first glance, the definition we have given of the Wagner coefficients of $\chi_t$ appears different than that which was given by Carlitz in \cite{Ca-40}, but it can be shown that they are indeed the same. For lack of space, we do not pursue this. Our definition is given with the intention of making a fast connection with Pellarin's series.
\end{rem}
We will show that for $t \in \mathbb{C}_\infty$ such that $|t| < q$ the formal Wagner series above converges for all $z \in \mathbb{C}_\infty$, and that it agrees with the image of the evaluation character $\chi_t(z)$ for $z \in A$. 

Taking these two facts for granted for the moment, the following theorem is quite interesting and relates the Wagner series above to $L(\chi_t,1)$. 

\begin{thm} \label{chicoeffs}
For all non-negative integers $d$ we have
$$\sum_{a \in A_+(d)}\frac{\chi_t(a)}{a} = (-1)^{d}\frac{b_d(\chi_t)}{L_d}.$$
\begin{proof}
Recall \ref{cor:affine} that 
$$M_d(\chi_t)(z) = \sum_{a \in A_+(d)}\chi_t(a)\frac{e_d(z-a)}{z-a}$$
is an affine polynomial in the variable $z$ over $A[t]$, with constant coefficient equal to $M_d(\chi_t)(0) = (-1)^d(D_d /L_d)b_d(\chi_t)$ by definition. 
We have 
$$\frac{M_d(\chi_t)(z)}{e_d(z - \theta^d)} = \sum_{a \in A_+(d)}\frac{\chi_t(a)}{z-a}.$$
Observing that $e_d(-\theta^d) = (-1)^{q^d}D_d$, evaluation at $z = 0$ yields
$$(-1)^{d+1} \frac{b_d(\chi_t)}{L_d} = -\sum_{a \in A_+(d)}\frac{\chi_t(a)}{a}.$$
\end{proof}
\end{thm}

\begin{rem}
Pellarin has shown, and one may verify immediately using Sheat's \cite{Sh98}, Theorem 1.4 (a), which characterizes the vanishing of the power sums $S_d(k)$, that
$$\sum_{a \in A_+(d)}\frac{\chi_t(a)}{a} = \frac{(-1)^d}{L_d}\prod_{j = 0}^{d-1}(t - \theta^{q^j}).$$
Thus by the proposition above we obtain a factorization over $A[t]$ of $b_{j}(\chi_t)$. Namely,
$$b_{j}(\chi_t) = \prod_{l = 0}^{j-1}(t - \theta^{q^l}).$$
The reader should compare this with the $L_i$. Indeed, 
$$\left.\frac{(-1)^{j-1}b_j(\chi_t)}{(t-\theta)L_{j-1}}\right|_{t = \theta} = 1.$$ 
Similarly, we have
$$b_j(\chi_t)(\theta^j) = D_j.$$
\end{rem}

We now proceed to show that the Wagner series for $\chi_t$ converges and agrees with $\chi_t$ on $A$. We prove two basic lemmas, the second of which motivates the definition we have given of the Wagner coefficients $b_i(\chi_t)$.

Let $B$ be an $A$-algebra that is an integral domain. We begin with the preliminary observation that any $\mathbb{F}_q$-linear polynomial $f \in B[z]$ of degree $q^d$ may be written as
$$f(z) = \sum_{i=0}^d \alpha_i e_i(z),$$
and the coefficients $\alpha_i \in B$ are uniquely determined by $f$. Indeed, the polynomials $e_i(z)$ are monic of degree $q^i$ in $z$, and hence the coefficients $\alpha_j$ may be found recursively from the coefficients of the powers $z^{q^i}$ of $f$.

Let $f \in L_{\mathbb{F}_q}(A,B)$. Recall that we have defined $P_d(f)(z)  = M_d(f)(z) - M_d(f)(0)$, and by Corollary \ref{cor:linear} this polynomial is $\mathbb{F}_q$-linear of degree at most $q^{d-1}$. For $d \geq 1$ and $0 \leq i < d$ we define the coefficients $\alpha_{d,i}(f)$ by the equality
\begin{eqnarray}\label{alphadef}
(-1)^{d}\frac{L_{d}}{D_{d}}P_d(f)(z) = \sum_{i = 0}^{d-1}\alpha_{d,i}(f)e_i(z).
\end{eqnarray}
\begin{lem} \label{lem:ccoeff1}
Let $f \in L_{\mathbb{F}_q}(A, B)$. For $d \geq 1$ we have the following recursive formula:
$$(-1)^{d+1}\frac{L_{d+1}}{D_{d+1}}P_{d+1}(f)(z) - \alpha_{d+1,d}(f)e_d(z) = (-1)^{d}\frac{L_d}{D_d}P_d(f)(z).$$
\begin{proof}
Again, we appeal to Lemma \ref{degreelemma}. Both the left-hand-side and right-hand-side of the equality above are non-constant polynomials with coefficients in the integral domain $B$ that have degree strictly less than $q^d$. For all $a \in A(d)$ we have $e_d(a) = 0$ by definition. Hence both sides are equal when evaluated on $A(d)$. Thus they are identically equal in $B[z]$.
\end{proof}
\end{lem}

\begin{lem} \label{lem:ccoeff2}
Let $f \in L_{\mathbb{F}_q}(A, B)$. Let $d \geq 0$. Then 
$$(-1)^d\frac{L_d}{D_d}M_d(f)(0) = D_d\alpha_{d+1,d}(f).$$
\begin{proof}
There are two cases:

First suppose $d=0$. Then 
$$(-1)^d\frac{L_d}{D_d}M_d(f)(0) = f(1),$$
and from \eqref{alphadef}
$$f(1) = (-1)^{d+1}\frac{L_{d+1}}{D_{d+1}}P_{d+1}(f)(1) = D_d\alpha_{d+1,d}(f).$$

Now suppose $d \geq 1$. Let $a \in A_+(d) \subseteq A(d+1)$. Then 
\begin{eqnarray*}
(-1)^{d+1}\frac{L_{d+1}}{D_{d+1}}P_{d+1}(f)(a) &=& f(a), \\ 
(-1)^d \frac{L_d}{D_d}M_d(f)(a) &=& f(a), \text{ and } \\
e_{d}(a) &=& D_d.
\end{eqnarray*} 
Hence by the previous lemma
$(-1)^d\frac{L_d}{D_d}M_d(f)(0) = D_d\alpha_{d+1,d}(f)$.
\end{proof}
\end{lem}

\begin{thm}
Let $t \in \mathbb{C}_\infty$ be such that $|t| < q$, then for all $z \in \mathbb{C}_\infty$ the sum
$$\sum_{j = 0}^\infty b_j(\chi_t) \frac{e_{j}(z)}{D_j}$$
converges and thus gives analytic interpolation to the quasi-character $\chi_t(z)$ originally defined on $A$.
\begin{proof}
For $t$ as above, it is easy to see that 
$$\sum_{a \in A_+(d)}\chi_t(a) a^{-1} \rightarrow 0 \text{ as } d \rightarrow \infty,$$ and hence $b_d(\chi_t) / L_d \rightarrow 0$ as $d \rightarrow \infty$ as well by Theorem \ref{chicoeffs}. This is exactly the condition in \cite{Go-89}, Theorem 3.1.10 required for convergence of the first expansion. 

To see that this series agrees with $\chi_t(z)$ for $z \in A$, let $a \in A(d)$. Then for all $j \geq d$ we have $e_j(a) = 0$. Hence on $A(d)$ the series agrees with the polynomial
$$\sum_{j = 0}^{d-1}b_j(\chi_t)\frac{e_j(z)}{D_j}.$$
By Lemmas \ref{lem:ccoeff1} and \ref{lem:ccoeff2} this is exactly $(-1)^d\frac{L_d}{D_d}P_d(\chi_t)(z)$, and by Corollary \ref{cor:linear} this agrees with $\chi_t$ on $A(d)$.
\end{proof}
\end{thm}

\begin{rem}
The interpolation of $\chi_t$ given above is not evaluation at $t$ away from $z \in A$. This is to be expected as $A$ sits discretely inside the completion with respect to the infinite place $K_\infty$ of the quotient field $K$ of $A$. We demonstrate this now. Let $t \in \mathbb{C}_\infty \setminus \overline{\mathbb{F}}_q$. Then by basic algebra there is a unique ring homomorphism from $K$ to $\mathbb{F}_q(t)$ which extends $\chi_t$, and this map is just evaluation of $f \in K$ at $t$. Let us call this map $\chi_t$ as well. Now, for $|t| < 1$ and $f \in K_\infty$ it makes sense to evaluate $f$ at $t$, and the function $f \mapsto f(t)$ is the unique continuous extension (in the variable $z$) of $\chi_t(z)$ to $K_\infty$, and hence satisfies the functional equation $\chi_t(fg) = \chi_t(f)\chi_t(g)$ for all $f,g \in K_\infty$. Now $\chi_t$ is also additive on $K_\infty$ and so its derivative is constant. Let us call its derivative $a_0$. Now using the functional equation $\chi_t(zw) = \chi_t(z)\chi_t(w)$ which is valid on all of $K_\infty$, and differentiating with respect to $z$ we obtain $a_0 w = a_0\chi_t(w)$, and hence if $a_0 \neq 0$ we conclude that $\chi_t$ is the identity function, a contradiction. Thus $\chi_t$ cannot have a power series representation which is valid on all of $K_\infty$ and for which $a_0 \neq 0$. As a corollary of this: since $\beta_{0}(\chi_t)$ can be shown to equal $1$, the extension of $\chi_t$ from $A$ to all of $\mathbb{C}_\infty$ given in the proposition above cannot be evaluation at $t$ on $K_\infty$.
\end{rem}

\bibliographystyle{plain}

\hfill \email{perkins@math.osu.edu}
\end{document}